\documentclass[10pt, article]{amsart}

\usepackage{tikz}
\usetikzlibrary{calc}
\usepackage{ae} % or {zefonts}
\usepackage[T1]{fontenc}
\usepackage[cp1250]{inputenc}
\usepackage{amsmath}
\usepackage{amssymb, amsfonts,amscd,verbatim}

\usepackage[normalem]{ulem}
\usepackage{hyperref}
\usepackage{indentfirst}
\usepackage{latexsym}
\input pstricks.tex
\input xy
\xyoption{all}

\usepackage{amsmath}    % need for subequations
\theoremstyle{plain}
\newtheorem{Pocz}{Poczatek}[section]
\newtheorem{Proposition}[Pocz]{Proposition}

\newtheorem{Theorem}[Pocz]{Theorem}
\newtheorem{Corollary}[Pocz]{Corollary}

\newtheorem{Lemma}[Pocz]{Lemma}
\newtheorem{Observation}[Pocz]{Observation}

\newtheorem{Question}[Pocz]{Question}

\newtheorem{Example}[Pocz]{Example}

\theoremstyle{definition}
\newtheorem{Definition}[Pocz]{Definition}

\theoremstyle{remark}
\newtheorem{Remark}[Pocz]{Remark}

\DeclareMathOperator*{\diam}{diam}

\def\asdim{\text{asdim}}

\numberwithin{equation}{section}
\title[Preserving coarse properties]
{Preserving coarse properties}

\author{J.~Dydak}
\address{University of Tennessee, Knoxville, TN 37996, USA}
\email{jdydak@utk.edu}

\author{\v Z. Virk}
\address{
Univerza v Ljubljani,
Jadranska ulica 19,
SI-1111 Ljubljana,
Slovenija }
\email{zigavirk@gmail.com}

\date{ \today
}


\thanks{This research was supported by the Slovenian Research
Agency grants P1-0292-0101 and J1-2057-0101.}

\begin{document}
\maketitle
\begin{center}
\today
\end{center}

\begin{abstract}
The aim of this paper is to investigate properties preserved and co-preserved by coarsely $n$-to-1 functions, in particular by the quotient maps $X\to X/\sim$ induced by a finite group $G$ acting by isometries on a metric space $X$. The coarse properties we are mainly interested in are related to asymptotic dimension and its generalizations: having finite asymptotic dimension, asymptotic Property C, straight finite decomposition complexity, countable asymptotic dimension, and metric sparsification property. We provide an alternative description of asymptotic Property C and we prove that the class of spaces with straight finite decomposition complexity coincides with the class of spaces of countable asymptotic dimension.
\end{abstract}

\section{Introduction}

The main topic of this paper is preservation and co-preservation of coarse properties
by certain classes of functions. The most important of them is the class of coarsely $n$-to-1
functions recently introduced by Miyata and Virk \cite{MV}. As shown in Section \ref{ASDIMOfFunctions}
that class is contained in the class of functions of asymptotic dimension $0$ introduced in \cite{BDLM}.

A class of functions $f:X\to Y$ preserves a coarse propery $\mathcal{P}$ if $f(X)$ has $\mathcal{P}$ whenever
$X$ has $\mathcal{P}$. A class of functions $f:X\to Y$ co-preserves a coarse propery $\mathcal{P}$ if $X$ has $\mathcal{P}$ whenever
$f(X)$ has $\mathcal{P}$. 

The coarse properties we are mainly interested in are related to asymptotic dimension and its generalizations: having finite asymptotic dimension, asymptotic Property C, straight finite decomposition complexity, countable asymptotic dimension, and metric sparsification property.

Besides investigating properties being preserved and co-preserved by coarsely $n$-to-1 functions and providing an alternative description of asymptotic Property C, our main result of the paper is that $X$ being of straight finite decomposition complexity is actually equivalent to $X$ having countable asymptotic dimension.

\section{Preliminaries}

One of the main ideas in topology is \textbf{approximating} general topological spaces $X$ by simplicial complexes. This is done by first selecting a cover $\mathcal{U} $ of $X$ and then constructing its \textbf{nerve} $N(\mathcal{U} )$. Recall $N(\mathcal{U})$ has $\mathcal{U} $ as its vertices and $[U_0,\ldots,U_n]$ is an $n$-simplex of $N(\mathcal{U})$ if $\bigcap\limits_{i=0}^n U_i\ne \emptyset$.

One of the most important characteristics of a simplicial complex $K$ is its \textbf{combinatorial dimension} $\dim(K)$, the supremum over all $n$ such that $K$ has an $n$-simplex. Therefore it makes sense to introduce the dimension of a family of subsets of a set $X$:
\begin{Definition}\label{DimOfCoverDef}
The \textbf{dimension} $\dim(U)$ of a family of subsets of a set $X$ is the combinatorial dimension of its nerve.
\end{Definition}

That leads to a concise explanation of the covering dimension $\dim(X)$ of a topological space: $\dim(X)\leq n$ if every open cover of $X$ can be refined by an open cover of dimension at most $n$.

Its dualization in coarse topology leads to the following definition:

\begin{Definition}\label{AsDimDef}
A metric space $X$ is of \textbf{asymptotic dimension} $asdim(X)$ at most $n$ if every uniformly bounded cover of $X$ can be coarsened to a uniformly bounded family of dimension at most $n$.
\end{Definition}

It turns out it makes sense to look at a metric space at different scales $R\ge 0$.
\begin{Definition}
Given a family of subsets $\mathcal{U} $ of metric space $X$ and a scale $R\ge 0$, the \textbf{dimension} $\dim_R(\mathcal{U} )$ of $\mathcal{U} $ at \textbf{scale} $R$ (or $R$-dimension in short) is the dimension of the family $B(\mathcal{U} ,R)$ of $R$-balls $B(U,R)$, $U\in \mathcal{U} $. Here $B(U,R)$ consists of $U$ and all points $x$ in $X$ such that $d_X(x,u) < R$ for some $u\in U$. The $R$-multiplicity of a cover $\mathcal{U}$ is defined to be $\dim_R(\mathcal{U})+1$.
\end{Definition}

\begin{Observation}[see \cite{Grom} or \cite{BellDranish A Hurewicz Type}]
$asdim(X)\leq n$ if and only if for each $R > 0$ there is a uniformly bounded cover $\mathcal{U} $ of $X$ such that $\dim_R(\mathcal{U} )\leq n$.
\end{Observation}

The easiest case of estimating $R$-dimension is in the case of unions of $R$-disjoint families:
\begin{Observation}\label{DisUnionsObs}
If ${\mathcal{U}}=\bigcup\limits_{i=0}^{n}{\mathcal{U}}_i$ and each $\mathcal{U}_i$ is $R$-disjoint (that means $d(x,y)\ge R$ for $x$ and $y$ belonging to different elements of $\mathcal{U}_i$), then $\dim_R(\mathcal{U})\leq n$.
\end{Observation}
\begin{proof}
It is sufficient to consider $R=0$ as $R$-balls of elements of $\mathcal{U}_i$ form a disjoint family for $R > 0$.
Notice each $x\in X$ belongs to at most one element of $\mathcal{U}_i$, hence it belongs to at most $(n+1)$ elements of $\mathcal{U}$.
\end{proof}

\begin{Lemma} \label{LemCoarseMapPullBackDisjointness}
 Suppose $f\colon X \to Y$ is coarse with control $E$. If $\mathcal V$ is an $E(d)$ disjoint collection of sets in $Y$ then $\{f^{-1}(V)\mid V\in\mathcal V\}$ is a $d$-disjoint collection.
\end{Lemma}

\begin{proof}
 If two points are at distance at most $d$ in $X$ then their images are at most $E(d)$ apart in $Y$ hence they cannot belong to different elements of $\mathcal V$.
\end{proof}

For technical reasons it is convenient to achieve the situation of Observation \ref{DisUnionsObs}:

\begin{Lemma}\label{LemmaMultToDisjointness}
Suppose $X$ is a metric space, $n\ge 1$, and $M, R > 0$.
If ${\mathcal{U}}=\{U_s\}_{s\in S}$ is a cover of $X$ of $R$-dimension at most $n$, then there is a cover ${\mathcal{V}} =\bigcup\limits_{i=0}^{n}{\mathcal{V}}^i$ of $X$
such that each ${\mathcal{V}}^i$ is $\frac{R}{n+1}$-disjoint. Furthermore:
\begin{enumerate}
  \item every element of $\mathcal{V}$ is contained in an intersection of at most $(n+1)$-many $R$-neighborhoods of elements of $\mathcal{U}$. In particular, for each finite subset $T$ of $S$ we have $B(W_T,-R/(2n+2))\subset \bigcap_{t\in T}B(U_t, R)$ (see the proof below for notation);
      \item if $\mathcal{U}$ is $M$-bounded then $\mathcal{V}$ is $(M+2R)$-bounded.
\end{enumerate}
\end{Lemma}

\begin{proof}
Define $f_s(x)=dist(x,X\setminus B(U_s,R))$. For each finite subset $T$ of $S$
define
$$
W_T=\{x\in X \mid  \min\{f_t(x)\mid  t\in T\} > \sup\{f_s(x) \mid  s\in S\setminus T\}\}=
$$
$$
\{x\in X \mid  f_t(x)>f_s(x) \mid  \forall t\in T, \forall s\in S\setminus T\}\}.
$$
Notice $W_T=\emptyset$ if $T$ contains at least $n+2$ elements.
Also, notice Fact 1: $W_T\cap W_F=\emptyset$ if both $T$ and $F$
are different but contain the same
number of elements.
Let us estimate the Lebesgue number of ${\mathcal{W}}=\{W_T\}_{T\subset S}$.
Given $x\in X$ arrange all non-zero values $f_s(x)$ from the largest
to the smallest. Add $0$ at the end and look at gaps between those values.
The largest number is at least $R$, there are at most $(n+1)$ gaps,
so one of them is at least $\frac{R}{n+1}$. That implies the ball $B(x,R/(2n+2))$
is contained in one $W_T$ ($T$ consists of all $t$ to the left of the gap) hence the Lebesgue number is at least $R/(2n+2)$.
Define ${\mathcal{V}}^i$ as
$$
\{B(W_T,-R/(2n+2))\}=\{\{x\in W_T \mid B(x,R/(2n+2))\subset W_T\}\},
$$
where $T$ ranges through all subsets of $S$ containing exactly $i$ elements. By Fact 1 above, each $\mathcal{V}^i$ is an $\frac{R}{n+1}$-disjoint family.

For every finite subset $T$ of $S$ we have $B(W_T,-R/(2n+2))\subset \bigcap_{t\in T}B(U_t, R)$ by the definition of $W_T$ as $f_t$ is nonzero on $W_T$ for every $t\in T$. This proves (1). Consequently, sets $W_T$ are $(M+2R)$-bounded  and so are their subsets $B(W_T,-R/(2n+2))$, elements of $\mathcal{V} $, proving (2).
\end{proof}

With the help of Lemma \ref{LemmaMultToDisjointness} we are going to show similarity of asymptotic property C to having countable asymptotic dimension
(see \ref{CountableAsdim}).

\begin{Definition} [Dranishnikov \cite{Dran AsyTop}]
A metric space $X$ has \textbf{asymptotic property C} if for every sequence $R_1 < R_2 < \ldots$ there exists $n\in \mathbb{N}$
such that $X$ is the union of $R_i$-disjoint families $\mathcal{U}_i$, $1\leq i\leq n$,
that are uniformly bounded.
\end{Definition}

\begin{Theorem}
 A metric space $X$ has asymptotic property C if and only if there is a sequence of integers $n_i\ge 0$, $i\ge 1$, such that for any sequence of positive real numbers $R_i$, $i\ge 1$, there is a finite sequence
$\mathcal{V}_i$, $i\leq n$, of uniformly bounded families of subsets of $X$ such that 
the dimension of $\mathcal{V}_i$ at scale $R_i$ is at most $n_i$ for $i\leq n$ and $X=\bigcup\limits_{i=1}^n \mathcal{V}_i$.
\end{Theorem}
\begin{proof}
 In one direction the proof is obvious: namely, $n_i=0$ for all $i$ works.
 
 Suppose there is a sequence of integers $n_i\ge 0$, $i\ge 1$, such that for any sequence of positive real numbers $R_i$, $i\ge 1$, there is a finite sequence
$\mathcal{V}_i$, $i\leq n$, of uniformly bounded families of subsets of $X$ such that 
the dimension of $\mathcal{V}_i$ at scale $R_i$ is at most $n_i$ for $i\leq n$ and $X=\bigcup\limits_{i=1}^n \mathcal{V}_i$.

Given an increasing sequence $M_i$ of positive real numbers, first define $m_j$ as $\sum\limits_{i\leq j}n_i$ and then
 define $R_i$ as $\sum\limits_{j \leq i}(n_j+1)\cdot M_{m_j+1}$.
Pick a finite sequence
$\mathcal{V}_i$, $i\leq n$, of uniformly bounded families of subsets of $X$ such that 
the dimension of $\mathcal{V}_i$ at scale $R_i$ is at most $n_i$ for $i\leq n$ and $X=\bigcup\limits_{i=1}^n \mathcal{V}_i$.
Using Lemma \ref{LemmaMultToDisjointness} decompose each family $\mathcal{V}_i$ into the union
$\mathcal{U}^j_i$, $0\leq j\leq n_i+1$, of uniformly bounded families that are $R_i/(n_i+1)$-disjoint.
Order the new families by lexicographic order: first look at the subscript index, then look at superscript index.
The result is a finite sequence of uniformly bounded families that together comprise a cover of $X$
and the $i$-th family is $M_i$-disjoint.
\end{proof}

\section{Bornologous functions}

\begin{Definition}\label{DefBornoMap}
A function $f: X \to Y$ of metric spaces is $C$-\textbf{bornologous}, where $C:[0,\infty)\to [0,\infty)$, if the image of every $r$-bounded set in $X$ is $C(r)$-bounded. $f$ is \textbf{bornologous} if it is $C$-\textbf{bornologous} for some function $C:[0,\infty)\to [0,\infty)$. $f$ is \textbf{$(a,b)-$Lipschitz} if it is $C$-bornologous for  $C(r)=ar+b$.

A function $f\colon X \to Y$ of metric spaces is \textbf{$R-$close} to $g\colon X \to Y$,  if $d(f(x),g(x))\leq R, \forall x\in X.$ $f$ is \textbf{close} to $g$ if it is $R-$close to $g$ for some $R>0$.

A bornologous function $f\colon X \to Y$ of metric spaces is a \textbf{coarse equivalence} if there exists a bornologous function $g\colon Y\to X$, so that $f \circ g$ is close to the identity $id_Y$ and that $g \circ f$ is close to the identity $id_X$.
\end{Definition}

\begin{Theorem}\label{ProjectionEquivThm}
Suppose $X$ is a metric space such that $d(x,y) \ge 1$ if $x\ne y$. If $\sim$ is an equivalence relation whose equivalence classes are uniformly bounded, then the projection $p:X\to X/\sim$ is a coarse equivalence if $X/\sim$ is equipped with the Hausdorff metric $d_H$.
\end{Theorem}
\begin{proof}
Suppose each equivalence class is of diameter less than $R$. Notice 
$$ d(x,y)-2R  \leq d_H([x],[y])\leq d(x,y)+2R$$
for all $x,y\in X$. That means that $p$ is $(1,2R)$-Lipschitz and any selection function $s:X/\sim\to X$ is $(1,2R)$-Lipschitz as well. Since $s\circ p$ is $R$-close to the identity $id_X$ and $p\circ s$ is the identity on $X/\sim$, both $p$ and $s$ are coarse equivalences.
\end{proof}

\section{Coarsely $n$-to-1 functions}

\begin{Definition}[Condition $B_n$ of \cite{MV}] \label{DefNto1Map}
A bornologous function $f: X \to Y$ of metric spaces is \textbf{coarsely $n$-to-1} (with \textbf{control} $C$) if there is a function $C: [0,\infty)\to [0,\infty)$ so that for each subset $B$ of $Y$ with $\diam(B) \leq r$, $f^{-1}(B) = \bigcup\limits_{i=1}^n A_i$ for some subsets $A_i$ of $X$ with $\diam(A_i) < C(r)$ for $i = 1, \ldots, n$.
\end{Definition}

See \cite{DV} for other conditions equivalent to $f$ being coarsely $n$-to-1.

An example of a coarse $n$-to-1 map is $z\mapsto z^n$ in the complex plane. Here is a more general case:

\begin{Example}
If a finite group $G$ acts on a metric space $X$ by bornologous functions, then the projection $p:X\to X/G$ is
coarsely $|G|$ to 1 if $X/G$ is given the Hausdorff metric.
\end{Example}
\begin{proof} Change the original metric $\rho$ on $X$ to
$$d(x,y):=\sum\limits_{g\in G} \rho(g\cdot x,g\cdot y)$$
and notice it is coarsely equivalent to $\rho$. Therefore, the Hausdorff metrics induced by both $d$ and $\rho$ on $X/G$ are coarsely equivalent. Notice $G$ acts on $X$ via isometries with respect to the metric $d$. That means we can reduce our proof to the case of $G$ acting on $X$ by isometries.

First, notice that $p\colon X \to X/G$ is $1$-Lipschitz, hence bornologous. Indeed, if $d(x,y) < r$ and $z\in G\cdot y$, then
$z=g\cdot y$ for some $g\in G$ and $d(g\cdot x,g\cdot y)=d(x,y) < r$. That means $z\in B(G\cdot x,r)$ and
$d(G\cdot x,G\cdot y)\leq d(x,y)$.

Given $x\in X$ and given $r > 0$ the point-inverse $p^{-1}(B(g\cdot x,r))$ is contained in $B(G\cdot x,r)$ which is clearly the union $\bigcup\limits_{y\in G\cdot x} B(y,r)$ of at most $|G|$ many sets of diameter at most $2\cdot r$.
\end{proof}

It is easy to check that the property of being coarsely $n$-to-1 is a coarse property: if two maps are coarsely equivalent and one of them is coarsely $n$-to-1 then the other one is coarsely $n$-to-1 as well. Furthermore, the composition of a coarsely $n$-to-1 map with a coarse equivalence (from the left or from the right) is coarsely $n$-to-1.

\begin{Definition}
A set $A$ is  \textbf{$R$-connected} if for every pair of points $x,y\in A$ there exist points $x_0=x, x_1, \ldots, x_k=y$ in $A$ for which $d(x_i,x_{i+1})\leq R$. An \textbf{$R$-component} is a maximal $R$-connected set.
\end{Definition}

\begin{Lemma}\label{FibersOfcoarselynTo1Functions}
If a function $f: X \to Y$ of metric spaces is coarsely $n$-to-1 with control $C$, then for each subset $B$ of $Y$ with $\diam(B) \leq r$, the number of $R$-components of $f^{-1}(B)$, $R\ge C(r)$, is at most $n$ and each $R$-component has diameter at most $2n\cdot R$.
\end{Lemma}
\begin{proof}
Given a subset $B$ of $Y$ with $\diam(B) \leq r$, express $f^{-1}(B)$ as
$\bigcup\limits_{i=1}^n A_i$ for some subsets $A_i$ of $X$ with $\diam(A_i) \leq C(r)$ for $i = 1, \ldots, n$.
Consider $R\ge C(r)$ and pick an $R$-component $A$ of $f^{-1}(B)$. Notice $A$ is the union of some sets among the family $\{A_i\}_{i=1}^n$. Also, every two points in $A$ can be connected by an $R$-chain of points. If that chain contains two points from the same set $A_i$, then it can be shortened by eliminating all points between them. That means there is an $R$-chain that has at most $2n$ points and $\diam(A)\leq 2n\cdot R$.
\end{proof}

\begin{Lemma}[Lemma 3.6 of \cite{MV}]\label{LemmaMiyataVirk}
Suppose $f: X \to Y$ is coarsely n to 1 with control $C$. Then for every cover $\mathcal{U}$ of $X$ and for every $r>0$ we have
	$$
\dim_r(f(\mathcal{U})) \leq (\dim_{C(r)}(\mathcal{U})+1) \cdot n-1
	$$
\end{Lemma}
\begin{proof}
Assume $\dim_{C(r)}(\mathcal{U})=m < \infty$.
Let $C$ be a control function of $f$. We may assume $C(r)\to\infty$ as $r\to\infty$.
 Let's count the number of elements of $f(\mathcal{U})$
that intersect a given set $B$ of diameter at most $r$. It can be estimated from above by the sum of the numbers of elements of $\mathcal{U} $ intersected by sets $A_i$, where $f^{-1}(B)$ is expressed as
$\bigcup\limits_{i=1}^n A_i$ for some subsets $A_i$ of $X$ with $\diam(A_i) \leq C(r)$ for $i = 1, \ldots, n$.  Each of those sets intersects at most $(m+1)$ elements of $\mathcal{U}$, so the total estimate for $f^{-1}(B)$ is $n\cdot (m+1)$. That proves $\dim_r(f(\mathcal{U})) \leq (\dim_{C(r)}(\mathcal{U})+1) \cdot n-1$.
\end{proof}

\begin{Corollary} [Miyata-Virk \cite{MV}]
Suppose $f: X \to Y$ is a surjective function  of metric spaces that is coarsely $n$-to-1 for some $n\ge 1$.
If $asdim(X)$ is finite, then $asdim(Y) \leq n\cdot (asdim(X)+1)-1$.
\end{Corollary}

\begin{Definition}
$f:X\to Y$ is \textbf{coarsely surjective} if $Y\subset B(f(X),R)$ for some $R > 0$.
\end{Definition}

\begin{Proposition}
[Structure of coarsely n to 1 maps]
Every coarsely $n$-to-1 map $f: X\to Y$ factors as $f=q\circ p$, where $p: X \to Z$ is a coarse equivalence, $q: Z \to Y$ is coarsely $n$ to $1$, and $q^{-1}(y)$ has at most $n$ points for each $y\in Y$.
\end{Proposition}

\begin{proof}
Given a metric $\rho$ on $X$ change it to the one defined by $d(x,y)=\max(1,\rho(x,y))$ if $x\ne y$, and note it is coarsely equivalent to $\rho$.
Using Lemma \ref{FibersOfcoarselynTo1Functions} find $R > 0$ so that $R$-components of fibers $f^{-1}(\{y\})$ of $f$
have diameter at most $2n\cdot R$ and there are at most $n$ of them.
Define an equivalence relation $\sim$ on $X$ as follows: $x\sim z$ if and only if $f(x)=f(z)$ and both $x$ and $z$
belong to the same $R$-component of their fiber.
By Theorem \ref{ProjectionEquivThm} the projection $p:X\to X/\sim$ is a coarse equivalence if $X/\sim$ is equipped with the Hausdorff metric. Obviously, there is $q:X/\sim\to Y$ such that $f=q\circ p$ and each fiber of $q$ has at most $n$ elements. As $p$ is a coarse equivalence, $q$ is bornologous and coarsely $n$-to-1.
\end{proof}

\begin{Proposition}\label{PropTransferCoverProperties}
Suppose $f: X \to Y$ is coarsely $n$-to-1 with control $D$ and coarse with control $E$. Suppose $\mathcal U$ is a cover of $X$.

\begin{enumerate}
	\item If $\mathcal U$ is $b$-bounded and of $D(r)$-dimension $m$ then $f(\mathcal U)$ 	is 	 $E(b)$-bounded of $r$-dimension at most $(m+1) \cdot n$;
	\item If $\mathcal U$ is $b$-bounded and of $D(r)$-dimension $m$ then there exists
	$(E(b)+r)$-bounded cover ${\mathcal{V}}=\bigcup\limits_{i=1}^{n (m+1)}{\mathcal{V}}^i$ of
 	$f(X)$ so that each $\mathcal V ^i$ is a $\frac{r}{(n \cdot (m+1))}$-disjoint family.
\end{enumerate}
\end{Proposition}

\begin{proof}
(1) follows from Lemma \ref{LemmaMiyataVirk}.

For (2) use (1) and Lemma \ref{LemmaMultToDisjointness}.
\end{proof}

\section{Asymptotic dimension of functions}\label{ASDIMOfFunctions}
The well-known Hurewicz Theorem for maps (also known as
Dimension-Lowering Theorem, see~\cite[Theorem 1.12.4 on
p.109]{Engel}) says $\dim(X)\leq \dim(f)+\dim(Y)$ if $f: X\to
Y$ is a closed map of separable metric spaces and $\dim(f)$ is
defined as the supremum of $\dim(f^{-1}(y))$, $y\in Y$. Bell and
Dranishnikov \cite{BellDranish A Hurewicz Type} proved a variant
of the Hurewicz Theorem for asymptotic dimension without defining
the asymptotic dimension of a function. However, Theorem 1 of
\cite{BellDranish A Hurewicz Type} may be restated as
$\asdim(X)\leq \asdim(f)+\asdim(Y)$, where $\asdim(f)$ is the
smallest integer $n$ such that $\asdim(f^{-1}(B_R(y)))\leq n$
uniformly for all $R
> 0$. 

\begin{Definition}
Given a function $f: X\to Y$ of metric spaces, its \textbf{asymptotic dimension} $\asdim(f)$ is the supremum of $\asdim(A)$
such that $A\subset X$ and $\asdim(f(A))=0$.
\end{Definition}

In \cite{BellDranish A Hurewicz Type} there is a concept of a
family $\{X_{\alpha}\}$ of subsets of $X$ satisfying
$\asdim(X_{\alpha}) \le n$ uniformly. Notice that in our language
this means there is one function that serves as an $n$-dimensional
control function for all $X_{\alpha}$.

\begin{Corollary}[Bell-Dranishnikov \cite{BellDranish A Hurewicz
Type}]\label{HurewiczTypeOfBellD} Let $f: X \to Y$ be a
Lipschitz function of metric spaces. Suppose that, for every $R
>0$,
$$\asdim\{f^{-1}(B_R (y))\} \le n$$
uniformly (in $y \in Y$). If $X$ is geodesic, then $\asdim(X) \le
\asdim(Y) +n$.
\end{Corollary}

\begin{Corollary}
If $f$ is a coarsely $n$-to-1 function, then $\asdim(f)=0$.
\end{Corollary}
\begin{proof}
Let $C$ be a control function of $f$.
By \ref{FibersOfcoarselynTo1Functions}, for each subset $B$ of $Y$ with $\diam(B) \leq r$ and for each $R > C(r)$, the number of $R$-components of $f^{-1}(B)$ is at most $n$ and each $R$-component has diameter at most $n\cdot R$. The last fact is sufficient to conclude $\asdim(f)\leq 0$.
\end{proof}

\section{Preservation of coarse properties}

The last several sections will be devoted to the issue of preservation of coarse invariants by coarsely $n$-to-1 functions: MSP,  finite decomposition complexity, countable asymptotic dimension, Asymptotic Property C, and Property A.

\begin{Theorem}\label{ThmAsDim}
 If $f:X\to Y$ is coarsely $n$-to-1, coarse, and coarsely surjective, then $asdim X \leq asdim Y\leq (asdim X +1)n-1$.
\end{Theorem}

\begin{proof}
The first inequality follows easily form Proposition \ref{PropTransferCoverProperties} (3) when using the definition of asymptotic dimension in terms of $n$-many $R$-disjoint uniformly bounded families covering the space. The second inequality is the main result of \cite{MV}.
\end{proof}

\begin{Theorem}
Suppose $f:X\to Y$ is coarsely $n$-to-1 and coarsely surjective. If $X$ has Asymptotic Property C, then $Y$ has Asymptotic Property C.
\end{Theorem}
\begin{proof}
We may assume $f$ to be surjective as all properties in question are coarse invariants. Pick functions 
$C, E: [0,\infty)\to [0,\infty)$ such that $d_Y(f(x),f(y))\leq E(d_X(x,y))$ for all $x,y\in X$
and for each subset $B$ of $Y$ with $\diam(B) \leq r$, the number of $C(r)$-components of $f^{-1}(B)$ is at most $n$ and each $C(r)$-component has diameter at most $2n\cdot C(r)$. We may assume $E(r)\to \infty$ as $r\to\infty$.

Suppose $X$ has asymptotic property C and $R_1<R_2<\ldots$. There exists $m$ and there are uniformly bounded families $\mathcal U_i, i=1,\ldots,m$ such that the family $\mathcal U = \bigcup\limits_{i=1}^m \mathcal U_i$ is a cover of $X$ and each $\mathcal{U}_i$ is $C(n\cdot R_{i\cdot n})$-disjoint.  By Proposition \ref{PropTransferCoverProperties} we have $f(\mathcal U_i)=\bigcup\limits_{j=1}^n \mathcal V_{i,j}$ where $\mathcal V_{i,j}$ is an $R_{i\cdot n}$-disjoint and uniformly bounded family for $j\in \{1,\ldots,n\}$. We have obtained a collection of $m\cdot n$ uniformly bounded families $\mathcal V_{i,j}$ covering $f(X)$. Furthermore, $\mathcal V_{i,j}$ is $R_{n\cdot (i-1)+j}$-disjoint as it is $R_{i\cdot n}$-disjoint and  $R_{n\cdot(i-1)+j} \leq R_{in}$. This proves the theorem.
\end{proof}

\begin{Theorem}\label{ThmPropC}\label{AsdimZeroCaseForPropC}
Suppose $f:X\to Y$ is coarsely surjective of asymptotic dimension $\asdim(f)=0$. If $Y$ has Asymptotic Property C, then $X$ has Asymptotic Property C.
\end{Theorem}

\begin{proof}
Pick a control function $E:[0,\infty)\to [0,\infty)$ of $f$.
Suppose $Y$ has asymptotic property C. Choose $R_1<R_2<\ldots$. There exist $m$ and uniformly bounded families $\mathcal U_i, i=1,\ldots,m$ such that the family $\mathcal U = \bigcup\limits_{i=1}^m \mathcal U_i$ is a cover of $Y$ and each $\mathcal{U}_i$ is $E(R_i)$-disjoint. Therefore each $f^{-1}(\mathcal{U}_i)$ is $R_i$-disjoint.
Since $\asdim(f)=0$, there is $M > 0$ such that each $R_m$-component of $f^{-1}(U)$, $U\in \mathcal{U}_i$ is of diameter at most $M$. If we define $\mathcal{V}_i$ as consisting of $R_m$-components of $f^{-1}(U)$, $U\in \mathcal{U}_i$, then each $\mathcal{V}_i$ is uniformly bounded, $R_i$-disjoint, and $X=\bigcup\limits_{i=1}^m \mathcal{V}_i$. That concludes the proof of $X$ having asymptotic property C.
\end{proof}

\begin{Question}
Suppose $f:X\to Y$ is coarsely surjective of finite asymptotic dimension $\asdim(f)$. If $Y$ has Asymptotic Property C, does $X$ have Asymptotic Property C?
\end{Question}

\begin{Corollary}
If $f:X\to Y$ is coarsely $n$-to-1 and coarsely surjective, then $X$ has Asymptotic Property C if and only if $Y$ has Asymptotic Property C.
\end{Corollary}

\section{Metric Sparsification Property }
\begin{Definition}(see \cite{CTWY})\label{MSPDef}
A metric space $X$ has MSP (\textbf{Metric Sparsification Property}) if for all $R > 0$ and for each positive $c < 1$ there exists $S > 0$ such that for all probability measures $\mu$ on $X$ there exists an $R$-disjoint family $\{\Omega_i\}_{i\ge 1}$ of subsets of $X$ of diameter at most $S$ satisfying
$$\sum\limits_{i=1}^\infty \mu(\Omega_i) > c.$$
\end{Definition}

\begin{Remark}
As noted in \cite{BNSWW} a metric space $X$ has MSP if and only if there is $c > 0$ such that for all $R > 0$ there exists $S > 0$ with the property that for all probability measures $\mu$ on $X$ there exists an $R$-disjoint family $\{\Omega_i\}_{i\ge 1}$ of subsets of $X$ of diameter at most $S$ satisfying
$$\sum\limits_{i=1}^\infty \mu(\Omega_i) > c.$$
\end{Remark}

\begin{Proposition}
If $\asdim(X)$ is finite, then $X$ has MSP.
\end{Proposition}
\begin{proof}
Apply \ref{LemmaMultToDisjointness} to detect, for each $R > 0$, a uniformly bounded cover that decomposes into a union of $(n+1)$ families, each $R$-disjoint.
Therefore, given a probability measure $\mu$ on $X$, one of those families adds up to a subset $\Omega$ of $X$ whose measure is at least $\frac{1}{n+1}$.
\end{proof}

\begin{Theorem}\label{ThmMSP}
If $f:X\to Y$ is coarsely $n$-to-1, $f$ is coarsely surjective, and $X$ has MSP, then $Y$ has MSP.
\end{Theorem}

\begin{proof}
We may assume $f$ to be surjective as all properties in question are coarse invariants. Suppose $f$ is coarsely $n$-to-$1$ with control $D$, and bornologous with control $E$. We have to find, for each $R$, an $S > 0$ such that for any probability measure $\mu$ on $Y$ there is an $R$-disjoint family $\Omega_j$ in $Y$ so that
$$\mu(\bigcup\limits_{j=1}^\infty \Omega_j) > 1/(2n)$$
and diameter of each $\Omega_j$ is at most $S$.
As $X$ has MSP there exists $B > 0$ so that for every probability measure $\lambda$ on $X$ there is a $D(n R)$-disjoint family $\Omega'_i$ in $X$ so that $$\lambda(\bigcup\limits_{i=1}^\infty \Omega'_i) > 1/2$$
and diameter of each $\Omega'_i$ is at most $B$. We will prove that $S=E(B)+n R$ suffices.

\textbf{Step 1:} Transferring a measure to $X$. Suppose $\mu$ is a probability measure on $Y$. For each $y\in Y$ choose (by surjectivity) $x_y \in f^{-1}(\{y\})$. Define a probability measure $\lambda$ on $X$ by
    $$
    \lambda (A)=\mu (\{y\in Y \mid x_y \in A\}), \quad \emph{ for every } A\subset X.
    $$
Note that $\lambda (A)=\mu(f(A)), \forall A \subset X$.
Choose a $D(n R)$-disjoint family $\Omega'_i$ in $X$ so that $$\lambda(\bigcup\limits_{i=1}^\infty \Omega'_i) > 1/2$$
and diameter of each $\Omega'_i$ is at most $B$. In particular, the $D(n R)$-dimension of the collection $\{\Omega'_i\}$ is at most $1$.

\textbf{Step 2:} Transferring a cover to $Y$. By Proposition \ref{PropTransferCoverProperties} (2) (for $\mu=1$) there exists an $(E(b)+n R)$-bounded cover ${\mathcal{V}}=\bigcup\limits_{i=1}^{n}{\mathcal{V}}^i$ of
$f(\bigcup\limits_{i=1}^\infty \Omega'_i)$ so that each $\mathcal V ^i$ is an $(n\cdot r)/n$-disjoint family. As $\mu(f(\bigcup\limits_{i=1}^\infty \Omega'_i))=1/2$ there exists $i_0$ so that $\mu(\mathcal V ^{i_0})\geq 1/(2n)$. This completes the proof as  $\{\Omega_j\}:=\mathcal V ^{i_0}$ works.
\end{proof}

\begin{Corollary}
Suppose $X$ and $Y$ are of bounded geometry. If $f:X\to Y$ is coarsely $n$-to-1, $f$ is coarsely surjective, and $X$ has Property A, then $Y$ has Property A.
\end{Corollary}

\begin{proof}
It was proved in \cite{BNSWW} that Property A is equivalent to MSP for spaces of bounded geometry. 
\end{proof}

\begin{Definition}
A bornologous function $f:X\to Y$ of metric spaces has \textbf{MSP} if $f^{-1}(A)$ has MSP for every subset $A$ of $Y$ of asymptotic dimension $0$.
\end{Definition}

\begin{Corollary}
If $\asdim(f)$ is finite or $f$ is coarsely $n$-to-$1$, then $f$ has MSP.
\end{Corollary}

\begin{Proposition}
$f:X\to Y$ has MSP if and only if for every $c, R, K > 0, c<1$ there is $S > 0$ such that for every probability measure $\mu$ on $X$ such that the diameter of $f(supp(\mu))$ is less than $K$
there is a subset $\Omega$ of $X$ whose $R$-components are $S$-bounded and $\mu(\Omega) > c$.
\end{Proposition}
\begin{proof}
Suppose there is $c, R, K > 0$ ($c < 1$) such that for every $n > 1$ there is a probability measure $\mu_n$ on $X$ with diameter of $f(supp(\mu_n))$ less than $K$ such that for any subset $\Omega$ of $X$ satisfying $\mu_n(\Omega) > c$ there is an $R$-component of $\Omega$ of diameter bigger than $n$.

$\bigcup\limits_{n > 1} f(supp(\mu_n))$ cannot be a bounded set and, by picking a subsequence of measures, we may achieve $A=\bigcup\limits_{n > 1} f(supp(\mu_n))$ being of asymptotic dimension $0$. Since $f^{-1}(A)$ has MSP, there is $S > 0$
such that for any measure $\mu$ on $X$ there is a subset $\Omega$ of $X$
whose $R$-components have diameter at most $n$ and $\mu(\Omega) > c$.
Pick $n > S$ and consider $\mu=\mu_n$. The set $\Omega$ picked for that measure
has an $R$-component of diameter bigger than $n$, a contradiction.
\end{proof}

\begin{Theorem}\label{ThmMSPReverse}
If $f:X\to Y$ has MSP and $Y$ has MSP, then $X$ has MSP.
\end{Theorem}
\begin{proof}
Suppose $R_X > 0$. Pick $R_Y > 0$ such that $d_X(x,y) \leq R_X$ implies $d_Y(f(x),f(y))\leq R_Y$. Choose $K > 0$ such that for any probability measure $\mu$ on $Y$ there is a subset $\Omega$ of measure bigger than $0.5$ whose $R_Y$-components are $K$-bounded. Pick $S > 0$ with the property that for every probability measure $\mu$ on $X$ such that the diameter of $f(supp(\mu))$ is less than $K$
there is a subset $\Omega$ of $X$ whose $R$-components are $S$-bounded and $\mu(\Omega) > 0.5$.

Given a probability measure $\mu$ on $X$ transfer it to $Y$ as follows: $\lambda(A)=\mu(f^{-1}(A)), \forall A\subset Y$. Find a subset $\Lambda$ of $Y$ satisfying $\lambda(\Lambda) > 0.5$ whose $R_Y$-components are $K$-bounded. Only countably many $\Lambda_i$ of those $R_Y$-components are of interest as the rest have measure $0$.
Given $i$ consider $f^{-1}(\Lambda_i)$ and find a subset $\Omega_i$ of it
whose $R_X$-components are $S$-bounded and $\mu(\Omega_i) > 0.5\cdot \mu(f^{-1}(\Lambda_i))$. Look at $\Omega=\bigcup\limits_i \Omega_i$ and notice its $R_X$-components are $S$ bounded and $\mu(\Omega) > 0.25$.
\end{proof}

\section{Countable asymptotic dimension}

Countable asymptotic dimension was introduced in \cite{JD1} as a generalization of the concept of
straight finite decomposition complexity introduced by Dranishnikov and Zarichnyi \cite{DZ}. One of the main results of this section is that actually the two concepts are equivalent.

A partition of a set is a covering by disjoint sets. We also introduce a notation: if $\mathcal{U}$ is a collection of subsets of $X$ and $A$ is a subset of $X$ then $A \cap \mathcal{U}=\{A \cap U \mid U \in \mathcal{U}\}$.

\begin{Definition}
$X$ is of \textbf{straight finite decomposition complexity} \cite{DZ}  if 
for any increasing sequence of positive real numbers $R_1 < R_2 < \ldots$ there a sequence
$\mathcal{V}_i$, $i \leq n$, of families of subsets of $X$ such that the following conditions are satisfied:
\begin{itemize}
\item[1.] $\mathcal{V}_1=\{X\}$,
\item[2.] each element $U\in \mathcal{V}_i$, $i < n$, can be expressed as a union of at most $2$ families from $\mathcal{V}_{i+1}$ that are $R_i$-disjoint,
\item[3.] $\mathcal{V}_n$ is uniformly bounded.
\end{itemize}
\end{Definition}

\begin{Definition}\label{CountableAsdim}
A metric space $X$ is of \textbf{countable asymptotic dimension} if there is a sequence of integers $n_i\ge 1$, $i\ge 1$, such that for any sequence of positive real numbers $R_i$, $i\ge 1$, there is a sequence
$\mathcal{V}_i$ of families of subsets of $X$ such that the following conditions are satisfied:
\begin{itemize}
\item[1.] $\mathcal{V}_1=\{X\}$,
\item[2.] each element $U\in \mathcal{V}_i$ can be expressed as a union of at most $n_i$ families from $\mathcal{V}_{i+1}$ that are $R_i$-disjoint,
\item[3.] at least one of the families $\mathcal{V}_i$ is uniformly bounded.
\end{itemize}
\end{Definition}

\begin{Corollary}
In the definition of spaces of countable asymptotic dimension we may assume each $\mathcal V_i$ to be a partition of $X$, i.e., a disjoint collection of subsets covering $X$.
\end{Corollary}

\begin{proof}
Use Lemma \ref{LemmaCountableAsDim} as an inductive step. The initial step for $i=1$ holds as $\{X\}$ is an obvious partition of itself.
\end{proof}

\begin{Theorem}
 Metric space $X$ is of countable asymptotic dimension if and only if it is of straight finite decomposition complexity (sFDC).
\end{Theorem}

\begin{proof}
 One direction is fairly simple by definition. Suppose $X$ is of countable asymptotic dimension. Choose $R_1<R_2<\ldots$. As $X$ is of countable asymptotic dimension we obtain a sequence of families $\mathcal V_i$ corresponding to $R_{n_1}<R_{n_1+n_2}<R_{n_1+n_2+n_3}<\ldots$. In particular, we may assume $\mathcal V_2$ consists of $R_{n_1}$-disjoint families $\mathcal W_1,\ldots, \mathcal W_{n_1}$ covering $X$. Let $W_i$ denote the union of all elements of $\mathcal W_i$. We will construct families $\mathcal U_i$ corresponding to the definition of the sFDC for $i\in\{1,\ldots,n_1\}$ inductively:
 \begin{itemize}
	\item[1.] $\mathcal U_1=\{X\}$;
	\item[2.] $\mathcal U_2=\mathcal W_1 \cup \bigcup _{i=2}^{n_1}W_i$. $X$, the only element 		of $\mathcal U_1$, is a 	union of at most two $R_1$-disjoint families in $\mathcal 			U_2$:
	\begin{itemize}
		\item[2a.] $\mathcal W_1$, which is $R_{n_1}$-disjoint by definition hence it is also 				 $R_1$-disjoint;
		\item[2b.] $\{ \bigcup _{i=2}^{n_1}W_i\}$ which is $R_1$-disjoint as it is a one-element collection.
\end{itemize}
	\item[3.] the inductive step is the following: for $m<n_1$ define $\mathcal U_m=\mathcal W_1 			 \cup \ldots \cup\mathcal W_{m-1} \cup \bigcup _{i=m}^{n_1}W_i$. Every element 		of $	\mathcal 	 U_{m-1}$ is a union of at most two $R_m$-disjoint families in $\mathcal U_m$: every element of $\mathcal W_1 \cup \ldots \cup\mathcal W_{m-2}$ appears in $U_m$ as well and $\bigcup _{i=m-1}^{n_1}W_i$ can be expressed as a union of two families:
		\begin{itemize}
		\item[3a.] $\mathcal W_{m-1}$, which is $R_{n_1}$-disjoint by definition hence it is also 				 $R_m$-disjoint;
		\item[3b.] $\{ \bigcup _{i=m}^{n_1}W_i\}$ which is $R_m$-disjoint as it is a one element 				 collection.
		\end{itemize}
	\item[4.] $\mathcal U_{n_1}=\mathcal W_1 \cup \ldots \cup\mathcal W_{n_1}$. Again, every element 	 of $	 \mathcal 	U_{n_1}$ is a union of at most two $R_{n_1}$-disjoint families in $\mathcal U_m$: every element of $\mathcal W_1 \cup \ldots \cup\mathcal W_{n_1-2}$ appears in $U_{n_1}$ as well and $\bigcup _{i=n_1-2}^{n_1}W_i$ can be expressed as a union of two families:
		\begin{itemize}
		\item[4a.] $\mathcal W_{n_1-1}$, which is $R_{n_1}$-disjoint by definition;
		\item[4b.] $\mathcal W_{n_1}$, which is $R_{n_1}$-disjoint by definition .						 
\end{itemize}	
\end{itemize}

We have thus obtained $\mathcal U_{n_1}=\mathcal V_2$. Proceed in the same way for every element of $\mathcal V_2$ to obtain $U_{i}$ for $i\in \{n_1+1,\ldots, n_1+n_2\}$ with $\mathcal U_{n_1+n_2}=\mathcal V_2$. By induction we eventually obtain a uniformly bounded family.
\end{proof}

\begin{Proposition}
If $X$ is of countable asymptotic dimension then for every $R$ there exists $n$ so that $X$ can be covered by $n-$many collections $\mathcal U_1, \ldots, \mathcal U_n$ of subsets of $X$, all of which are uniformly bounded and $R$-disjoint.
\end{Proposition}

\begin{proof}
Suppose $X$ is of countable asymptotic dimension and choose covers $\mathcal V_i$ corresponding to sequence $R_i:=R+i$. By the definition we eventually obtain a collection $\mathcal V_i$, which can be decomposed into a finite number of uniformly bounded $R_i$-disjoint families, hence they are also $R$-disjoint, which suffices.
\end{proof}

\begin{Lemma}\label{LemmaCountableAsDim}
Suppose $\mathcal{U}_i, \mathcal{V}_{i+1}$ and $\mathcal{V}_{i+2}$ are covers of space $X$ with the following properties.
\begin{enumerate}
  \item $\mathcal{U}_i$ is a partition of $X$;
  \item each element of $\mathcal{U}_i$ is contained in a union of at most $n_i$ families from $\mathcal{V}_{i+1}$ that are $R_i$-disjoint;
  \item each element of $\mathcal{V}_{i+1}$ can be expressed as a union of at most $n_{i+1}$ families from $\mathcal{V}_{i+2}$ that are $R_{i+1}$-disjoint.
\end{enumerate}
Then there exists a partition $\mathcal{U}_{i+1}$ of $X$ so that:
\begin{description}
  \item[a] each element of $\mathcal{U}_i$ can be expressed as a union of at most $n_i$ families from $\mathcal{U}_{i+1}$ that are $R_i$-disjoint;
  \item[b] each element of $\mathcal{U}_{i+1}$ is contained in a union of at most $n_{i+1}$ families from $\mathcal{V}_{i+2}$ that are $R_{i+1}$-disjoint;
  \item[c] each element of $\mathcal{U}_{i+1}$ is contained in some element of $\mathcal{V}_{i+1}$.
\end{description}
\end{Lemma}

\begin{proof}
Choose $U \in \mathcal{U}_i$ and let $\mathcal{W}_1, \mathcal{W}_2, \ldots, \mathcal{W}_{n_i}$ denote a collection of $R_i$-disjoint families from $\mathcal{V}_{i+1}$ so that $\bigcup_{i=1}^{n_i}\bigcup \mathcal{W}_i$ contains $U$. Define 
$\widetilde{\mathcal{W}_i}=\{U \cap W \mid W \in \mathcal{W}_i\}$ for every $i$ and note that 
$\widetilde{\mathcal{W}_1}, \widetilde{\mathcal{W}_2}, \ldots, \widetilde{\mathcal{W}_{n_i}}$ is a collection of $R_i$-disjoint families for which $\widetilde {\mathcal{W}}=\bigcup_{i=1}^{n_i}\bigcup \widetilde{\mathcal{W}_i}$ equals $U$. The collection of subsets $\bigcup_{i=1}^{n_i}\widetilde{\mathcal{W}_i}$ may be well-ordered in the form $\{W_j\}_{j\in J}$ where $J$ is a well-ordered set. Define a collection $\{U_j\}_{j\in J}$ by the rule
$$
\forall j\in J: \quad U_j =W_j \setminus \bigcup_{k < j} W_k.
$$
Note that $\{U_j\}_{j\in J}$ is a partition of $U$. Undo the well-ordering by reindexing sets $\{U_j\}_{j\in J}$ back into a collection of $R_i$-disjoint families $\widehat{\mathcal{W}_1}, \widehat{\mathcal{W}_2}, \ldots, 
\widehat{\mathcal{W}_{n_i}}$ which together constitute a partition of $U$: the reindexing should be exactly the inversion to previous well-ordering (i.e., if $W\in \widetilde{\mathcal{W}_1}$ was given index $j_0\in J$ then $U_{j_0}$ should belong to $\widehat {\mathcal{W}_1}$ ) and $R_i$-disjointness is preserved as we have only decreased the sets. Let $\mathcal{U}_U=\bigcup_{k=1}^{n_i}\widehat{\mathcal{W}_k}$ denote a collection of obtained sets.

Since $\mathcal{U}_i$ is a partition of $X$ and $\mathcal{U}_U$ is a partition of $U$ for every $U\in \mathcal{U}_i$, the collection $\mathcal{U}_{i+1}=\bigcup_{U\in \mathcal{U}_i} \mathcal{U}_U$ of subsets is a partition of $X$. Furthermore, each element of $\mathcal{U}_i$ can be expressed as a union of at most $n_i$ families from $\mathcal{U}_{i+1}$ that are $R_i$-disjoint by  construction. This proves \textbf{a}.

To prove \textbf{c} note that every element of $\mathcal{U}_{i+1}$ was obtained by taking an intersection of some element of $\mathcal{V}_{i+1}$ by some sets. In particular, every element of $\mathcal{U}_{i+1}$ is contained in some element of $\mathcal{V}_{i+1}$.

Statement \textbf{b} follows from \textbf{c} and (3).
\end{proof}

\begin{Theorem}\label{ThmCountableAsdim}
 Suppose $f:X\to Y$ is coarsely $n$-to-1 with control $D$, coarsely surjective and coarse with control $E$.  Then $Y$ is of countable asymptotic dimension if and only if $X$ is of countable asymptotic dimension.
\end{Theorem}

\begin{proof}
Suppose $Y$ is of countable asymptotic dimension.
Choose $R_1<R_2<\ldots$. According to the definition of the countable asymptotic dimension choose for $E(R_1)<E(R_2)<\ldots$ a sequence $\mathcal V_i$ of families of subsets of $Y$. Define $\mathcal U_i=\{f^{-1}(V)\mid V \in \mathcal V_i\}$. Sequence $\mathcal U_i$ (actually its finite subsequence, see (3) below in the proof) proves $X$ to be of countable asymptotic dimension:
\begin{enumerate}
  \item $\mathcal U_1=\{X\}$.
  \item if $V\in \mathcal V_i$ can be expressed as a union of at most $n_i$-many $E(R_i)$-disjoint families from $\mathcal V_{i+1}$ then $f^{-1}(V)$ can be expressed as a union of at most $n_i$-many $R_i$-disjoint families from $\mathcal U_{i+1}$ by taking the preimages and applying Lemma \ref{LemCoarseMapPullBackDisjointness}.
  \item suppose $\mathcal V_m$ is uniformly bounded by $b$. We will redefine $\mathcal U_{m+1}$. For every $V \in \mathcal V_m$ the preimage $f^{-1}(V)$ can be expressed as a disjoint union of at most $n$-many $(nD(b)+(n-1)R_i)-$bounded $R_i$-disjoint subsets of $X$. Let  $\mathcal U_{m+1}$ consist of all such sets. The collection $\mathcal U_{m+1}$ is uniformly bounded by $D(b)$ and every element of $\mathcal U_{m}$ can be expressed as a union of at most $n$-many elements of $\mathcal U_{m+1}$. This concludes the proof. For the sake of formal argument we may define $\mathcal U_{m+1}=\mathcal U_j, \forall j\geq m+1$ and note that the sequence of integers for $X$ may be taken to be $(\max\{n_i,n\})_i$.
\end{enumerate}
\bigskip

Suppose $X$ is of countable asymptotic dimension. We may assume $f$ to be surjective: the justification is a simple exercise. Choose $R_1<R_2<\ldots$. We will define a sequence $\mathcal{V}_i$ of families of subsets of $Y$ satisfying the conditions in the definition of the countable asymptotic dimension for parameters $R_1<R_2<\ldots$.

According to the definition of the countable asymptotic dimension for $X$ there exists a sequence $\{n_i\}_{i \geq 1}$ such that for $D(n n_1 R_1)<D(n n_2 R_2+2n n_1 R_1)<\ldots$ (the pattern of increasing parameters is not yet visible; however, it will become apparent that appropriate parameters may be chosen depending on $\{n_i\}, \{R_i\}$ and $n$) there exists a sequence $\mathcal U_i$ of partitions of $Y$ with appropriate properties. Define $\mathcal{V}_1=\{Y\}$.

Collection $\mathcal{U}_2$ is a partition of $X$ of $D(n n_1 R_1)$-multiplicity at most $n_1$. By Lemma \ref{LemmaMiyataVirk} the collection $f(\mathcal{U}_2)$ is a cover of $Y$ of $(n n_1 R_1)$-multiplicity at most $n n_1$. By Lemma \ref{LemmaMultToDisjointness} there exists a cover $\mathcal{V}_2$ of $Y$ consisting of $(n n_1)$-many $R_1$-disjoint families with the following property: 
\medskip

\begin{minipage}[t]{.97\textwidth}
every element of $\mathcal{V}_2$ is contained in the $(n n_1 R_1)$-neighborhood \\ of some element of $f(\mathcal{U}_2)$. \hfill (*)
\end{minipage}

\medskip

The next step: the definition of $\mathcal{V}_3$ is actually an inductive step in the construction of sequence $\mathcal{V}_i$. For the sake of simplicity we only present it for the case $i=3$. 

Choose any $U \in \mathcal{U}_2$. Recall that $U$ may be covered by a collection of $n_2$-many $D(n n_2 R_2+2n n_1 R_1)$-disjoint families $\mathcal{W}_1, \mathcal{W}_2, \ldots, \mathcal{W}_{n_2}$ (some of these families may be empty in order to get exactly $n_2$-many of them). Let $\mathcal{W}=\bigcup_{j=1}^{n_2}\mathcal{W}_j$ denote the corresponding partition of $U$. By Lemma \ref{LemmaMiyataVirk} the collection $f(\mathcal{W})$ is a cover of $f(U)$ of $(n n_2 R_2+2n n_1 R_1)$-multiplicity at most $n n_2$. We now expand the sets by $n n_1 R_1$ as suggested by (*): the collection of $(n n_1 R_1)$-neighborhoods of family $\mathcal{W}$ is a cover of the $(n n_1 R_1)$-neighborhood of $f(U)$ of $(n n_2 R_2)$-multiplicity at most $n n_2$. By Lemma \ref{LemmaMultToDisjointness} there exists a cover $\mathcal{V}_3^U$ of the $(n n_1 R_1)$-neighborhood of $f(U)$ consisting of $(n n_2)$-many $R_2$-disjoint families with the following property:
\medskip

\begin{minipage}[t]{.97\textwidth}
every element of $\mathcal{V}_3^U$ is contained in $(n n_2 R_2+ n n_1 R_1)$-neighborhood \\of some element of $f(\mathcal{U}_3)$. \hfill (**)
\end{minipage}

\medskip 
 
 Define a collection
    $$
    \mathcal{V}_3=\bigcup_{U\in \mathcal{U}_2} \mathcal{V}_3^U.
    $$

Observe the following properties of obtained $\mathcal{V}_3$:
\begin{description}
  \item[a] The collection $\mathcal{V}_3$ covers $Y$ as every $\mathcal{V}_3^U$ covers $f(U)$ for every $U\in \mathcal{U}_2$ and the collection $\{f(U)\}_{U\in \mathcal{U}_2}$ covers $Y$;
  \item[b] Every element of $\mathcal{V}_2$ is by (*) contained in $(n n_1 R_1)$-neighborhood of $f(U)$ for some element $U \in \mathcal{U}_2$. By construction $\mathcal{V}_3^U$ is a cover of the $(n n_1 R_1)$-neighborhood of $f(U)$ consisting of $(n n_2)$-many $R_2$-disjoint families. In particular: every element of $\mathcal{V}_2$ can be covered by a collection of $(n n_2)$-many $R_2$-disjoint families from $\mathcal{V}_3$.
\end{description}

Note that \textbf{a.} and \textbf{b.} are conditions required by the definition of the countable asymptotic dimension. 

Proceed by induction: use the (*)-type conditions (i.e., conditions stating "every element of $\mathcal{V}_i$ is contained in the $L(\{n_i\}, \{R_i\},n)$-neighborhood of some element of $f(\mathcal{U}_{i})$." Examples are condition (*) and condition (**) for all $U$.) to continue the sequence $D(n n_1 R_1)<D(n n_2 R_2+2n n_1 R_1)<\ldots$ of disjointness of $\mathcal{U}_i$: the sequence depends only on $\{n_i\}, \{R_i\}$ and $n$. Note that if $\mathcal{U}_i$ is uniformly bounded then the obtained $\mathcal{V}_i$ is (again, by the (*)-type condition) uniformly bounded as well, which completes the proof.
\end{proof}

\end{document}